\documentclass[14pt]{amsart}

\usepackage[cp1251]{inputenc}
\usepackage[ukrainian,english]{babel}
\usepackage{graphicx}%
\usepackage{multirow}%
\usepackage{amsmath,amssymb,amsfonts}%
\usepackage{amsthm}%
\usepackage{mathrsfs}%
\usepackage[title]{appendix}%
\usepackage{xcolor}%
\usepackage[hidelinks]{hyperref}

\usepackage{graphicx}
\usepackage{tikz-cd}
\usetikzlibrary{shapes,arrows}
\tikzstyle{block} = [rectangle, draw,
text width=7em, text centered, rounded corners,draw=black, minimum height=2.5em,text centered]
\tikzstyle{line} = [draw, -latex', thick]
\tikzset{node distance = 2cm}

\theoremstyle{thmstyleone}%
\newtheorem{theorem}{Theorem}[section]
\newtheorem{proposition}[theorem]{Proposition}%

\newtheorem{remark}[theorem]{Remark}%
\newtheorem{lemma}[theorem]{Lemma}%
\newtheorem{corollary}[theorem]{Corollary}
\newtheorem{principle}[theorem]{Principle}

\theoremstyle{definition}
\newtheorem{definition}[theorem]{Definition}%
\newtheorem{analogy}{Analogy}

\begin{document}

\title[Faithfulness and fractal (quasi-)equivalence principles for expansions]{Faithfulness and fractal (quasi-)equivalence principles for Perron, Engel, and Pierce expansions}

\author{Mykola Moroz}
\address{\emph{M.Moroz}: Department of Dynamical Systems and Fractal Analysis, Institute of Mathematics of NAS of Ukraine, Tereschenkivska 3, 01024 Kyiv, Ukraine}

\begin{abstract}
We establish several unifying principles that clarify the fractal properties of classical number expansions, which are generalized by the Perron expansions. In particular, we prove the fractal equivalence principle for the positive and alternating Perron expansions, the fractal quasi-equivalence principle for the~classical and modified Engel expansions, and the fractal quasi-equivalence principle for the Pierce expansions in the Perron and traditional notations. These results explain several known analogies and show that the Hausdorff dimension of sets defined by one expansion often coincides with that for another. The proofs rely on faithful families of coverings. In addition to deriving a~range of known theorems as direct corollaries of previous results, our approach yields new fractal properties of the Engel and Pierce expansions and provides a systematic framework for transferring Hausdorff dimension properties between different expansions.
\end{abstract}



\maketitle

\tableofcontents

\section{Introduction}\label{sec1}

Fractal properties of number expansions form a central topic in modern metric number theory. The positive and alternating Perron expansions, introduced and investigated in \cite{Moroz2025,Moroz2024}, generalize many classical constructions. They include as special cases the positive and alternating L\"{u}roth expansions, the classical and modified Engel expansions, the~DKB-expansion, as well as the Pierce, Sylvester, and restricted Oppenheim expansions. The guiding idea behind these generalizations is not merely to provide new expansions of real numbers, but to uncover reasons for systematic analogies between metric and fractal theories associated with different expansions.

In particular, \cite{Moroz2025} established that the metric theories of the positive and alternating Perron expansions are equivalent: a digit-preserving bijection also preserves the Lebesgue measure, allowing results for one expansion to be deduced directly from the other. For example, several well-known theorems by~R\'{e}nyi~\cite{Renyi1962-1} and Shallit~\cite{Shallit1986} become direct corollaries of each other when viewed through this unified lens. However, while the metric equivalence of these expansions is now well understood, a corresponding theory for fractal properties has not yet been systematically developed.

The aim of this paper is to fill this gap by establishing fractal analogues of these equivalence results in metric theory. More precisely, we introduce and prove three principles that clarify the relationships between the fractal theories of Perron, Engel, and Pierce expansions. These principles not only explain a number of known analogies, but also allow one to transfer results on Hausdorff dimensions from one expansion to another without repeating lengthy proofs.

The main results of this paper are the establishment of the fractal equivalence principle for Perron expansions and the fractal quasi-equivalence principles for Engel and Pierce expansions. The first principle states that problems concerning the Hausdorff dimension of sets defined via the alternating Perron expansion are equivalent to the~corresponding problems for the positive Perron expansion. This principle applies, for instance, to the~modified Engel and Pierce expansion. The second and third principles allow the systematic transfer of numerous fractal properties of the~classical Engel expansion to its modified version, and hence to the Pierce expansions in traditional and Perron notations. 

These principles allow researchers to focus on establishing only those fractal properties of the modified Engel expansions that cannot be deduced from the classical ones using the proposed method. Furthermore, computing the Hausdorff dimension for sets defined via the Pierce expansion becomes unnecessary: each such problem reduces to an analogous one for the modified Engel expansion.

This paper is organized as follows. Section 2 introduces the basic definitions and notation related to Perron expansions. Section 3 discusses faithful families of coverings generated by Perron expansions. Section 4 presents the main results of this paper: the fractal equivalence principle for Perron expansions, the fractal quasi-equivalence principle for Engel expansions, and the fractal quasi-equivalence principle for Pierce expansions. Section 5 revisits several classical fractal results for the Pierce expansion, deriving them as a corollary of earlier results for the~classical or the modified Engel expansions using the proposed principles. Finally, Section 6 establishes new properties of the modified Engel and Pierce expansions by combining known properties of the classical Engel expansion with new fractal principles.

\section{Preliminaries}
We first provide the basic definitions and properties of Perron expansions.

\begin{definition}[\cite{Moroz2025,Moroz2024}]
	A \textit{Perron expansion} of $x\in(0,1]$ is a representation of one of the following two forms:
	\begin{itemize}
		\item \textbf{Positive Perron expansion}
		\begin{gather}\label{PosPerronSer}
			x=\frac{r_0}{p_1}+\sum_{n=1}^{\infty}\frac{r_0\cdots r_n}{(p_1-1)p_1\cdots(p_n-1)p_n p_{n+1}},
		\end{gather}
		\item \textbf{Alternating Perron expansion}
		\begin{gather}\label{AltPerronSer}
			x=\frac{r_0}{q_1-1}+\sum_{n=1}^{\infty}\frac{(-1)^n r_0\cdots r_n}{(q_1-1)q_1\cdots(q_n-1)q_n(q_{n+1}-1)},
		\end{gather}
	\end{itemize}
	where $(r_n)_{n=0}^{\infty}$, $(p_n)_{n=1}^{\infty}$, and $(q_n)_{n=1}^{\infty}$ are sequences of natural numbers satisfying $$p_n \geq r_{n-1}+1 \quad\text{and}\quad q_n \geq r_{n-1}+1 \quad (n\in\mathbb{N}).$$
\end{definition}

Fix a sequence $P=(\varphi_n)_{n=0}^\infty$ of functions, where $\varphi_0\in\mathbb{N}$ is constant and $\varphi_n\colon\mathbb{N}^n \to \mathbb{N}$ for $n\in\mathbb{N}$.

\begin{definition}[\cite{Moroz2024}]
	If $r_0=\varphi_0$ and $r_n=\varphi_n(p_1,\ldots,p_n)$ for $n\in\mathbb{N}$, then the positive Perron expansion~\eqref{PosPerronSer} is called the \textit{$P$-representation} (or $P$-expansion) of $x$ and is denoted by $\Delta^P_{p_1 p_2 \ldots}$.
\end{definition}

\begin{definition}[\cite{Moroz2025}]
	If $r_0=\varphi_0$ and $r_n=\varphi_n(q_1,\ldots,q_n)$ for $n\in\mathbb{N}$, then the alternating Perron expansion~\eqref{AltPerronSer} is called the \textit{$P^-$-representation} (or $P^-$-expansion) of $x$ and is denoted by $\Delta^{P^-}_{q_1 q_2 \ldots}$.
\end{definition}

For any sequence $P$, every $x\in(0,1]$ has a unique $P$-representation and at most one $P^-$-representation (see~\cite{Moroz2025,Moroz2024}). If $x=\Delta^P_{p_1 p_2\ldots}$, then $p_n=p_n(x)$ is called the \textit{$n$th $P$-digit} of $x$. If $x=\Delta^{P^-}_{q_1 q_2\ldots}$, then $q_n=q_n(x)$ is called the \textit{$n$th $P^-$-digit} of $x$.

\begin{definition}[\cite{Moroz2025,Moroz2024}]
	For natural numbers $c_1,\ldots,c_k$ satisfying 
	$c_i \geq \varphi_i(c_1,\ldots,c_{i-1})+1$ for $i\leq k$, the~set 
	$$\Delta^P_{c_1\ldots c_k}=\{x\in(0,1]\colon p_1(x)=c_1,\ldots,p_k(x)=c_k\}$$ is called the \emph{$P$-cylinder of rank $k$ with base $c_1\ldots c_k$}, and the set $$\Delta^{P^-}_{c_1\ldots c_k}=\{x\in(0,1]\colon q_1(x)=c_1,\ldots,q_k(x)=c_k\}$$ is called the \emph{$P^-$-cylinder of rank $k$ with base $c_1\ldots c_k$}.
\end{definition}

The set $IS^{P^-}$ of all numbers from $(0,1]$ that do not have a $P^-$-representa\-tion is countable and consists precisely of the infima and suprema of $P^-$-cylinders:
$$IS^{P^-}=\left\{x\in(0,1]\colon x=\inf\Delta^{P^-}_{c_1\ldots c_k} \text{ ~or~ }  x=\sup\Delta^{P^-}_{c_1\ldots c_k}\text{ for some }P^-\text{-cylinder }\Delta^{P^-}_{c_1\ldots c_k}\right\}.$$

\begin{proposition}[\cite{Moroz2025,Moroz2024}]
	Each $P$-cylinder has the form $(a,b]$. Each $P^-$-cylinder has the form $\mbox{$(a,b)\setminus IS^{P^-}$}$. The~$P$-cylinder and the $P^-$-cylinder with the same base $c_1\ldots c_k$ have the same diameter, given by
	\begin{gather}
		|\Delta^P_{c_1 \ldots c_k}|=|\Delta^{P^-}_{c_1 \ldots c_k}|=\frac{r_0 \cdots r_{k-1}}{(c_1-1)c_1\cdots(c_k-1)c_k},\label{cyl}
	\end{gather}
	where $|\cdot|$ denotes the diameter of the set, $r_0=\varphi_0$, and $r_n=\varphi_n(c_1,\ldots,c_n)$ for all $n=1,\ldots,k-1$.
\end{proposition}

For some sequences $P$, the positive and alternating Perron expansions reduce to well-known classical expansions:
\begin{itemize}
	\item the positive and alternating L\"utoth expansions, if $\varphi_n\equiv 1$ for all $n\in\mathbb{N}\cup\{0\}$:
	\begin{gather*}
		\frac{1}{p_1}+\frac{1}{(p_1-1)p_1p_2}+\frac{1}{(p_1-1)p_1(p_2-1)p_2 p_3}+\cdots, \\
		\frac{1}{q_1-1}-\frac{1}{(q_1-1)q_1(q_2-1)}+\frac{1}{(q_1-1)q_1(q_2-1)q_2(q_3-1)}-\cdots. 
	\end{gather*}
	\item the classical and alternating Engel expansions, if $\varphi_n(x_1,\ldots,x_n)=x_n-1$ for all $n\in\mathbb{N}$ with $\varphi_0=1$:
	\begin{gather}
		\frac{1}{p_1}+\frac{1}{p_1p_2}+\frac{1}{p_1 p_2 p_3}+\cdots,\label{ClassEngel} \\
		\frac{1}{q_1-1}-\frac{1}{q_1(q_2-1)}+\frac{1}{q_1 q_2(q_3-1)}-\cdots.\notag 
	\end{gather}
	\item the modified Engel and Pierce expansions, if $\varphi_n(x_1,\ldots,x_n)=x_n$ for all $n\in\mathbb{N}$ with $\varphi_0=1$:
	\begin{gather}
		\frac{1}{p_1}+\frac{1}{(p_1-1)p_2}+\frac{1}{(p_1-1)(p_2-1)p_3}+\cdots,  \label{ModEngel} \\ 
		\frac{1}{q_1-1}-\frac{1}{(q_1-1)(q_2-1)}+\frac{1}{(q_1-1)(q_2-1)(q_3-1)}-\cdots.\label{PiercePerronNotation}
	\end{gather}
	\item the Sylvester and second Ostrogradsky expansions, if $\varphi_n(x_1,\ldots,x_n)=x_n(x_n-1)$ for all $n\in\mathbb{N}$ with $\varphi_0=1$:
	\begin{gather*}
		\frac{1}{p_1}+\frac{1}{p_2}+\frac{1}{p_3}+\cdots, \label{Sylvester}\\
		\frac{1}{q_1-1}-\frac{1}{q_2-1}+\frac{1}{q_3-1}-\cdots. \label{SecondOstr}
	\end{gather*}
\end{itemize}

Note that the Pierce expansion in the Perron notation \eqref{PiercePerronNotation} is slightly different from its traditional form given by 
\begin{equation}\label{PierceTraditionalNotation}
	\frac{1}{\widetilde{q}_1}-\frac{1}{\widetilde{q}_1\widetilde{q}_2}+\cdots+\frac{(-1)^{n-1}}{\widetilde{q}_1\cdots\widetilde{q}_n}+\cdots.
\end{equation}
The digits $q_n$ of the Pierce expansion in the Perron notation \eqref{PiercePerronNotation} are greater by one than the corresponding digits $\widetilde{q}_n$ in the traditional notation \eqref{PierceTraditionalNotation}. The same remark applies to the second Ostrohradsky expansion.

Moreover, if $\varphi_0=1$ and $$\varphi_n(x_1,\ldots,x_n)=\varphi_n(x_n)=\frac{a_n(x_n)}{b_n(x_n)}\cdot x_n(x_n-1)\in\mathbb{N}$$ for all $n\in\mathbb{N}$, where $a_n\colon \mathbb{N}\to\mathbb{N}$ and $b_n\colon\mathbb{N}\to\mathbb{N}$, then the positive Perron expansion coincides with the restricted Oppenheim expansion:
\begin{gather*}
	\frac{1}{p_1}+\frac{a_1}{b_1}\cdot\frac{1}{p_2}+\cdots+\frac{a_1\cdots a_n}{b_1\cdots b_n}\cdot\frac{1}{p_{n+1}}+\cdots.\label{Oppenheim}
\end{gather*}

Almost all known expansions of real numbers can be represented as $I\text{-}F$-expansions introduced in \cite{GNT2017}. We briefly recall their construction. Let $(N_k)_{k=1}^\infty$ be a sequence of sets such that $N_k=\{0,1,\ldots,n_k\}$ for some $n_k\in\mathbb{N}$ or $N_k=\mathbb{N}\cup\{0\}$.

The unit interval (rank 0) is split into at most countable number consecutive intervals (cylinders) $\Delta^{I\text{-}F}_{c_1}$ of rank 1, where $c_1\in N_1$. Each interval of rank 1 is then split into at most countable number consecutive intervals $\Delta^{I\text{-}F}_{c_1 c_2}$ of rank 2, where $c_2\in N_2$. This process is continued inductively under the assumption that
\begin{equation*}
	\lim\limits_{n\to\infty}\left|\Delta^{I\text{-}F}_{c_1\ldots c_n}\right|=0
\end{equation*} 
for every sequence $(c_n)_{n=1}^\infty$ with $c_n\in N_n$ for all $n\in\mathbb{N}$. If there exists a sequence $(c_n(x))_{n=1}^\infty$ such that
\begin{equation*}
	x=\bigcap_{n=1}^\infty \Delta^{I\text{-}F}_{c_1(x)\ldots c_n(x)}:=\Delta^{I\text{-}F}_{c_1(x) c_2(x)\ldots},
\end{equation*} 
then the expression $\Delta^{I\text{-}F}_{c_1(x) c_2(x)\ldots}$ is
called the $I\text{-}F$-expansion of $x\in[0,1]$.

Here, $I\in [0,1]$ is a parameter whose $n$th binary digit determines the direction in which cylinders of rank $n-1$ are split into cylinders of rank $n$. If the $n$th binary digit of $I$ equals 1, the splitting proceeds from left to right; otherwise, it proceeds from right to left. If $I$ has two different binary expansions, then we choose the one containing the digit 1 in the period. 

In particular, the positive Perron expansions correspond to $I\text{-}F$-expansions with $I=0$, while the alternating Perron expansions correspond to $I\text{-}F$-expansions with
\begin{equation*}
	I=\frac{1}{2^2}+\frac{1}{2^4}+\frac{1}{2^6}+\cdots=\frac{1}{3}.
\end{equation*} 
Note that in the Perron expansions, the cylinder enumeration is shifted compared to the standard $I\text{-}F$-expansions; this shift does not affect the geometric structure of the cylinders, only their indices. Next, by studying the Perron expansions, we will use some general properties of the $I\text{-}F$-expansions.

\section{Faithful families of coverings generated by Perron expansions}

In this section, we recall the basic definitions of faithful families of coverings and present some auxiliary facts that will be used to prove the main results.

Calculating the Hausdorff dimension is often challenging due to the need to consider a broad class of covering sets. To overcome this difficulty, the notion of faithfulness was introduced in \cite{AILT2020,AKNT2016} (arXiv version of paper \cite{AILT2020} was published in 2013) and subsequently employed in \cite{LZh2017,LZh2016,SZhL2017,TV2025}. This notion allows one to work with narrower, yet technically convenient, classes of covering sets when calculating the Hausdorff dimension. We begin by recalling some basic and auxiliary definitions from \cite{AILT2020,AKNT2016}, incorporating slight generalizations for technical convenience.

\begin{definition}
	Let $\Phi$ be a family of subsets of $\Omega$, where $\Omega\subset[0,1]$. 
	The family $\Phi$ is called a \emph{fine family of coverings} on $\Omega$ if for every $\varepsilon>0$ there exists a countable (or finite) $\varepsilon$-covering $\{E_j\}$ of $\Omega$ with $E_j\in\Phi$.
\end{definition}

\begin{definition}
	Let $\Phi$ be a fine family of coverings on $\Omega$.  
	The \emph{Hausdorff $\alpha$-dimensional measure of a set $E\subset\Omega$ with respect to $\Phi$} is defined by
	$$H^\alpha(E,\Phi)=\lim_{\varepsilon\to 0}\left(\inf_{|E_j|\leq\varepsilon}\sum_{j}|E_j|^\alpha \right),$$
	where the infimum is taken over all countable (or finite) $\varepsilon$-coverings $\{E_j\}$ of $E$ with $E_j\in\Phi$. The \emph{Hausdorff dimension of $E$ with respect to $\Phi$} is defined as
	$$\dim_H(E,\Phi)=\inf\{\alpha\colon H^\alpha(E,\Phi)=0\}.$$
\end{definition}

\begin{definition}
	A fine family of coverings $\Phi$ is called a \emph{faithful family of coverings} for the Hausdorff dimension calculation on $\Omega$ if
	\[
	\dim_H(E,\Phi)=\dim_H(E)
	\]
	for every $E\subset\Omega$, where $\dim_H(E)$ denotes the classical Hausdorff dimension.
\end{definition}

It is well known that the family of all binary subintervals of $[0,1]$ is faithful for the Hausdorff dimension calculation on the unit interval. However, this does not hold for more specialized families: in general, neither the family of all $P$-cylinders nor the family of all $P^-$-cylinders is faithful. For example, in {\cite[Theorem~2.2, Corollary~2.8]{AKNT2016}} and {\cite[Theorem 2.2]{SZhL2017}}, the authors proved that the families of cylinders generated by the~positive L\"{u}roth expansion and the restricted Oppenheim expansion, both of which are particular cases of the positive Perron expansion, are not faithful. Consequently, it is common to supplement the family of all cylinders with certain specific unions of cylinders so that the resulting family becomes faithful for the~Hausdorff dimension calculation. This approach has been successfully applied to particular cases of Perron expansions, including the~Engel expansion~\cite{BHP2023}, the Pierce expansion \cite{BPT2013}, the~positive L\"{u}roth expansion \cite{ZhPr2012}, and the restricted Oppenheim expansion \cite{SZhL2017}.

Let $\hat{\Phi}(I\text{-}F)$ be the family of all unions of consecutive $I\text{-}F$-cylinders of the same rank within a single $I\text{-}F$-cylinder of the previous rank, and let $D(I\text{-}F)$ be the set of numbers in $[0,1]$ that have an $I\text{-}F$-expansion.

\begin{proposition}[{\cite[Theorem 2.1]{GNT2017}}]\label{GNTfaithfulness}
	If there exists a constant $\alpha>1$ such that 
	\begin{equation}\label{faithfulness}
		\frac{1}{\alpha}\leq \frac{\left|\Delta^{I\text{-}F}_{c_1\ldots c_n k} \right|}{\left|\Delta^{I\text{-}F}_{c_1\ldots c_n (k+1)} \right|}\leq \alpha
	\end{equation}
	for all $n\in\mathbb{N}$ and all $k$ such that $k, k+1 \in N_{n+1}$, then $\hat{\Phi}(I\text{-}F)$ is a faithful family of coverings for the Hausdorff dimension calculation on $D(I\text{-}F)$.
\end{proposition}

Let $\mathfrak{P}$ be the family of all unions of consecutive $P$-cylinders of the same rank within a single $P$-cylinder of the previous rank. That is, $\mathfrak{P}$ comprises all sets of the following forms:
\begin{align*}
	\bigcup_{\substack{i=n\\
			n\geq r_k+1}}^m \Delta^P_{c_1\ldots c_{k}i},\qquad\qquad\bigcup_{\substack{i=n\\
			n\geq r_k+1}}^\infty \Delta^P_{c_1\ldots c_{k}i},
\end{align*}
where $r_k=\varphi_k\left(c_1,\ldots,c_k\right)$. The family $\mathfrak{P^-}$ is defined analogously to $\mathfrak{P}$ as the family of all unions of consecutive $P^-$-cylinders of the same rank within a single $P^-$-cylinder of the previous rank.

\begin{theorem}\label{theoremdovirchistdim}
	The families $\mathfrak{P}$ and $\mathfrak{P}^-$ are faithful families of coverings for the Hausdorff dimension calculation on $(0,1]$ and $(0,1)\setminus IS^{P^-}$, respectively.
\end{theorem}

\begin{proof}
	Since the arguments for the families $\mathfrak{P}$ and $\mathfrak{P^-}$ are identical, we present the proof only for $\mathfrak{P}$.
	
	By Proposition \ref{GNTfaithfulness}, it suffices to verify that there exists a constant $\alpha>1$ such that
	\begin{equation}\label{faithfulnessP}
		\frac{1}{\alpha}\leq \frac{\left|\Delta^{P}_{c_1\ldots c_n k} \right|}{\left|\Delta^{P}_{c_1\ldots c_n (k+1)} \right|}\leq \alpha
	\end{equation}
	for all $n\in\mathbb{N}$, all admissible sequences $c_1,\ldots,c_n$, and all $k\geq \varphi_n(c_1,\ldots,c_n)+1$. From \eqref{cyl}, we obtain
	\begin{equation*}
		\frac{\left|\Delta^{P}_{c_1\ldots c_n k} \right|}{\left|\Delta^{P}_{c_1\ldots c_n (k+1)} \right|}=\frac{k-1}{k+1}.
	\end{equation*}
	Hence, inequality \eqref{faithfulnessP} holds with $\alpha=3$.
\end{proof}

\section{Main results: new fractal principles for Perron, Engel, and Pierce expansions}

In this section, we prove the main results of the paper: the fractal equivalence principle for Perron expansions and the fractal quasi-equivalence principles for Engel and Pierce expansions.

\subsection{Fractal equivalence principle for the positive and alternating Perron expansions}

For the positive and alternating Perron expansions defined by a sequence of functions $P=\{\varphi_n\}_{n=0}^\infty$, consider the function $\mathcal{F}_P\colon (0,1]\to(0,1)\setminus IS^{P^-}$, given by
\begin{equation*}
	\mathcal{F}_P(\Delta^{P}_{c_1 c_2\ldots})=\Delta^{P^-}_{c_1 c_2\ldots}.
\end{equation*}

The function $\mathcal{F}_P$ is interesting from several perspectives. For example, it was shown in \cite{Moroz2025} that $\mathcal{F}_P$ preserves the Lebesgue measure. Furthermore, we have a well-founded conjecture that $\mathcal{F}_P$ is nowhere monotonic, has jump discontinuities at points of the countable set $IS^{P^-}$, and is continuous elsewhere. Therefore, whether $\mathcal{F}_P$ preserves the Hausdorff dimension remains far from trivial. The differentiability of $\mathcal{F}_P$ at points of continuity is still an open problem. In this article, we do not answer these questions, as they are outside the main topic of the present investigation. However, we do not rule out the possibility of discussing these properties in detail in future articles.

\begin{theorem}\label{preserveTh1}
	The function $\mathcal{F}_P$ preserves the Hausdorff dimension on $(0,1]$, i.e., $$\dim_H\bigl(\mathcal{F}_P(E)\bigr)=\dim_H (E)$$ for every set $E\subset(0,1]$.
\end{theorem}

\begin{proof}
	Since 
	$$\mathcal{F}_P(\Delta^{P}_{c_1\ldots c_k})=\Delta^{P^-}_{c_1\ldots c_k}\qquad \text{and}\qquad|\Delta^{P^-}_{c_1\ldots c_k}|=|\Delta^{P}_{c_1\ldots c_k}|,$$
	it follows that for every $M\in\mathfrak{P}$ we have 
	$$\mathcal{F}_P(M)\in\mathfrak{P^-}\qquad\text{and}\qquad |\mathcal{F}_P(M)|=|M|.$$ 
	Let $\{M_j\}$ be a countable (or finite) cover of $E\subset(0,1]$ by sets from $\mathfrak{P}$. Then $\{\mathcal{F}_P(M_j)\}$ forms a cover of $\mathcal{F}_P(E)$ by sets from $\mathfrak{P^-}$, and 
	$$\sum_{j}|\mathcal{F}_P(M_j)|^\alpha=\sum_{j}|M_j|^\alpha, ~~\alpha\geq 0.$$ 
	Conversely, every cover of $\mathcal{F}_P(E)$ by sets from $\mathfrak{P^-}$ arises in this way. Consequently,
	\begin{gather*}
		H^\alpha (\mathcal{F}_P(E),\mathfrak{P^-})=H^\alpha (E,\mathfrak{P}),\\
		\dim_H(\mathcal{F}_P(E),\mathfrak{P^-})=\dim_H(E,\mathfrak{P}),
	\end{gather*}
	and hence
	\begin{equation*}
		\dim_H \mathcal{F}_P(E)=\dim_H E.\qedhere
	\end{equation*}
\end{proof}

Thus, we obtain the following principle.

\begin{principle}[Fractal equivalence principle for the Perron expansions.]
	Let the positive and alternating Perron expansions be determined by the same sequence $P$. Then, for every set $\mathfrak{M}\subset\mathbb{N}^\mathbb{N}$, 
	\begin{gather*}
		\dim_H\left\{x\in(0,1]\colon (p_n(x))_{n=1}^\infty\in \mathfrak{M}\right\}=\dim_H\left\{x\in(0,1)\setminus IS^{P^-}\colon (q_n(x))_{n=1}^\infty\in \mathfrak{M}\right\},
	\end{gather*}
	where $p_n(x)$ and $q_n(x)$ denote the $n$th digits of the positive and alternating Perron expansions of $x$, respectively.
\end{principle}

In particular, the fractal equivalence principle applies to pairs of expansions such as the positive and alternating L\"{u}roth expansions, the modified Engel and Pierce expansions, as well as the Sylvester and second Ostrogradsky expansions. In all such cases, the alternating expansions are considered in their Perron notations, which makes it possible to apply the above principle.

This principle can be interpreted as a manifestation of the fractal and metric phenomenon known as the $G$-isomorphism of $I\text{-}F$-expansions (see \cite{GNT2017}). The essence of $G$-isomorphism is that if the digit-preserving transformation of one $I\text{-}F$-expansion into another transforms some faithful covering family $K$ to a faithful family $K'$ and preserves diameters of sets from $K$, then such a transformation preserves both Lebesgue measure and Hausdorff dimension. In this case, such $I\text{-}F$-expansions are called $G$-isomorphic, which implies their fractal and metric equivalence.

The fractal and metric equivalence of the positive and alternating L\"{u}roth expansions was discussed in \cite{GNT2014}. In contrast, the fractal equivalence for the modified Engel and Pierce expansions, as well as for the Sylvester and second Ostrogradsky expansions, follows from the fractal equivalence principle for Perron expansions and, to the best of our knowledge, has not been previously discussed.

\subsection{Fractal quasi-equivalence principle for the classical and modified Engel expansions}

We now consider two cases of the positive Perron expansion: the classical and modified Engel expansions.

For the sequence $P=(\varphi_n)_{n=0}^\infty$ given by $$\varphi_0=1\qquad \text{and}\qquad \varphi_n(x_1,\ldots,x_n)=x_n-1,$$ the positive Perron expansion reduces to the classical Engel expansion ($E$-expansion). The diameter of an $E$-cylinder $\Delta^{E}_{c_1 \ldots c_k}$ equals
\begin{gather*}
	|\Delta^{E}_{c_1 \ldots c_k}|=\frac{1}{c_1\cdots c_{k-1}c_k(c_k-1)}.
\end{gather*}
In this case, for every $x\in(0,1]$, the $E$-digit sequence $(p_n(x))_{n=1}^\infty$ is non-decreasing and satisfies $p_1(x)\geq 2$. Moreover, every non-decreasing sequence $(c_n)_{n=1}^\infty$ of natural numbers with $c_1\geq 2$ can be realized as the $E$-digit sequence of some $x\in(0,1]$. For the classical Engel expansion, we denote the faithful family $\mathfrak{P}$ by $\mathfrak{P}_E$.

For the sequence $P=(\varphi_n)_{n=0}^\infty$ given by $$\varphi_0=1,\qquad \text{and}\qquad\varphi_n(x_1,\ldots,x_n)=x_n,$$ 
the positive Perron expansion reduces to the modified Engel expansion ($E_\mathrm{mod}$-expansion). The diameter of an~$E_\mathrm{mod}$-cylinder $\Delta^{E_\mathrm{mod}}_{c_1 \ldots c_k}$ equals
\begin{gather*}
	|\Delta^{E_\mathrm{mod}}_{c_1 \ldots c_k}|=\frac{1}{(c_1-1)\cdots(c_{k-1}-1)(c_k-1)c_k}.
\end{gather*}
In this case, for every $x\in(0,1]$, the $E_\mathrm{mod}$-digit sequence $(p'_n(x))_{n=1}^\infty$ is strictly increasing with $p'_1(x)\geq 2$. Similarly, every strictly increasing sequence $(c'_n)_{n=1}^\infty$ of natural numbers with $c'_1\geq 2$ can be realized as the~$E_\mathrm{mod}$-digit sequence of some $x\in(0,1]$. For the modified Engel expansion, we denote the faithful family $\mathfrak{P}$ by $\mathfrak{P}_{E_\mathrm{mod}}$.

Let $x=\Delta^{E}_{c_1 c_2\ldots }$. Consider the function $\mathcal{T}\colon (0,1]\to (0,1]$ defined by
$$\mathcal{T}(x)=\mathcal{T}(\Delta^{E}_{c_1 c_2\ldots})=\Delta^{E_\mathrm{mod}}_{c'_1 c'_2\ldots},$$
where $c'_n=c_n+n-1$ for all $n\in\mathbb{N}$. 

Basic properties of similar functions were studied in \cite{Moroz2024}. In fact, $\mathcal{T}$ is a projection between $\overline{P}$-representations, i.e., the difference-based forms of the corresponding positive Perron expansions. This can be verified by expressing both $\Delta^{E}_{c_1 c_2 \ldots}$ and $\Delta^{E_{\mathrm{mod}}}_{c'_1 c'_2 \ldots}$ in their difference-based forms (see \cite{Moroz2024}). The function $\mathcal{T}$ is continuous and strictly increasing (see \cite[Lemma 6, Theorem 3]{Moroz2024}).

From its definition, $\mathcal{T}$ satisfies:
\begin{itemize}
	\item if $\mathcal{T}(x)=x'$, then $p'_n(x')=p_n(x)+n-1$ for all $n\in\mathbb{N}$;
	\item for every $y'\in(0,1]$, there exists a unique  $y\in(0,1]$ such that $y'=\mathcal{T}(y);$
	\item $\mathcal{T}(\Delta^{E}_{c_1\ldots c_k})=\Delta^{E_\mathrm{mod}}_{c'_1\ldots c'_k}$;
	\item if $U\in\mathfrak{P}_E$, then $\mathcal{T}(U)\in\mathfrak{P}_{E_\mathrm{mod}}$;
	\item for every $U'\in\mathfrak{P}_{E_\mathrm{mod}}$, there exists $U\in\mathfrak{P}_{E}$ such that $U'=\mathcal{T}(U)$.
\end{itemize}

Since $\mathcal{T}$ modifies the digits of an expansion, a set defined by some property of $(p_n(x))_{n=1}^\infty$ in the classical Engel expansion will generally not correspond to a set with the same property in the modified Engel expansion. In fact,
$$\mathcal{T}\left(\left\{x\in(0,1]\colon (p_n(x))_{n=1}^\infty\in\mathfrak{M}\right\}\right)=\left\{x\in(0,1]\colon (p'_n(x)-n+1)_{n=1}^\infty\in\mathfrak{M}\right\}.$$

\begin{lemma}\label{LemmaLipschitz}
	The function $\mathcal{T}$ is a Lipschitz transformation.
\end{lemma}

\begin{proof}
	Since any open interval $(a,b)\subseteq(0,1)$ can be represented as a countable (or finite) union $\bigcup \Delta^{E}_{c_1\ldots c_k}$ of pairwise disjoint $E$-cylinders, the length of this interval is given by $\sum\left|\Delta^{E}_{c_1\ldots c_k}\right|$.
	Moreover, the interval $\mathcal{T}((a,b))=(\mathcal{T}(a),\mathcal{T}(b))$ can be represented as a union of pairwise disjoint $E_\text{mod}$-cylinders, 
	$$\bigcup \Delta^{E_\mathrm{mod}}_{c'_1 \ldots c'_k}=\bigcup \mathcal{T}\left(\Delta^{E}_{c_1 \ldots c_k}\right),$$
	so that $$\mathcal{T}(b)-\mathcal{T}(a)=\sum\left|\mathcal{T}\left(\Delta^{E}_{c_1 \ldots c_k}\right)\right|=\sum\left|\Delta^{E_\mathrm{mod}}_{c'_1 \ldots c'_k}\right|.$$
	It remains to show the existence of a constant $M$ such that for any finite non-decreasing sequence $(c_n)_{n=1}^k$,
	\begin{gather}\label{M}
		\frac{|\mathcal{T}\left(\Delta^{E}_{c_1 \ldots c_k}\right)|}{|\Delta^{E}_{c_1 \ldots c_k}|}=\frac{|\Delta^{E_\mathrm{mod}}_{c'_1 \ldots c'_k}|}{|\Delta^{E}_{c_1 \ldots c_k}|}<M.
	\end{gather}
	Since $c_n\geq 2$, we have
	\begin{gather*}
		\frac{|\Delta^{E_\mathrm{mod}}_{c'_1 \ldots c'_k}|}{|\Delta^{E}_{c_1 \ldots c_k}|}=\frac{c_1\cdots c_k(c_k-1)}{(c'_1-1)\cdots(c'_k-1)c'_k}=\frac{c_1}{c_1-1}\cdot\ldots\cdot\frac{c_k}{c_k+k-2}\cdot \frac{c_k-1}{c_k+k-1}<2.
	\end{gather*}
	Thus, inequality \eqref{M} holds with $M=2$. This proves the lemma.
\end{proof}

\begin{corollary}\label{ineqdim}
	For every set $E\subseteq(0,1]$, we have
	$$\dim_H \mathcal{T}(E)\leq\dim_H E.$$
\end{corollary}

However, $\mathcal{T}$ is not bi-Lipschitz since there is no positive constant $m$ such that $$\frac{|\Delta^{E_\mathrm{mod}}_{c'_1 \ldots c'_k}|}{|\Delta^{E}_{c_1 \ldots c_k}|}>m.$$
Indeed, if $c_1=\cdots=c_k=2$, then this ratio equals $\frac{2^k}{(k+1)!}$, which tends to zero as $k\to\infty$. Below, we state sufficient conditions ensuring that the transformation $\mathcal{T}$ preserves the Hausdorff dimension of the set $E$.

For a positive function $\psi\colon \mathbb{N}\to\mathbb{R}^+$, we define the set $$\mathfrak{A}_\psi=\left\{x\in(0,1]\colon p_n(x)\geq\psi(n)\text{ ~for all sufficiently large } n\right\}.$$

\begin{theorem}\label{preserveTh2}
	If $\sum_{n=1}^{\infty}\frac{n}{\psi(n)}<\infty$, then for every set $E\subset\mathfrak{A}_\psi$, we have $$\dim_H \mathcal{T}(E)=\dim_H E.$$
\end{theorem}

\begin{proof}
	For each $k\in\mathbb{N}$, define the set $\mathfrak{A}_\psi^k$ by $$\mathfrak{A}_\psi^k=\left\{x\in(0,1]\colon p_n(x)\geq\psi(n)\text{ for all }n\geq k\right\},$$ 
	and $E^k=E\cap \mathfrak{A}_\psi^k$. Then $$\mathfrak{A}_\psi=\bigcup_{k=1}^\infty \mathfrak{A}_\psi^k,~~~~~E=\bigcup_{k=1}^\infty E^k.$$
	
	Consider an at most countable cover $\left\{U_i\right\}$ of $E^k$ by sets from $\mathfrak{P}_E$. If $U_i$ consists of $E$-cylinders of rank $n<k$ and $\Delta^{E}_{c_1 \ldots c_n}\subset U_i$, then we have
	\begin{align*}
		\frac{|\mathcal{T}(\Delta^{E}_{c_1 \ldots c_n})|}{|\Delta^{E}_{c_1 \ldots c_n}|}&=\frac{|\Delta^{E_\mathrm{mod}}_{c'_1 \ldots c'_n}|}{|\Delta^{E}_{c_1 \ldots c_n}|}=\frac{c_1}{c_1-1}\cdot\ldots\cdot\frac{c_n}{c_n+n-2}\cdot \frac{c_n-1}{c_n+n-1}\\
		&>1\cdot\frac{2}{2}\cdot\frac{2}{3}\cdot\ldots\cdot\frac{2}{n}\cdot\frac{1}{n+1}=\frac{2^{n-1}}{(n+1)!}\geq\frac{1}{k!},
	\end{align*}
	and hence
	$$\frac{|\mathcal{T}(U_i)|}{|U_i|}>\frac{1}{k!}.$$
	If $U_i$ consists of $E$-cylinders of rank $n\geq k$ and $\Delta^{E}_{c_1\ldots c_{n}}\subset U_i$, without loss of generality, assume that $c_m\geq\psi(m)$ for all $m$ with $k\leq m\leq n$ (otherwise such cylinders do not intersect $E^k$ and can be excluded). Then
	\allowdisplaybreaks
	\begin{align*}
		\frac{|\mathcal{T}(\Delta^{E}_{c_1 \ldots c_n})|}{|\Delta^{E}_{c_1 \ldots c_n}|}&=\frac{|\Delta^{E_\mathrm{mod}}_{c'_1 \ldots c'_n}|}{|\Delta^{E}_{c_1 \ldots c_n}|}=\frac{c_1}{c_1-1}\cdot\ldots\cdot\frac{c_{k-1}}{c_{k-1}+k-3}\cdot\frac{c_k}{c_k+k-2}\cdot\ldots\cdot\frac{c_n}{c_n+n-2}\cdot \frac{c_n-1}{c_n+n-1}\\
		&\geq1\cdot \frac{2}{2}\cdot\frac{2}{3}\cdot\ldots\cdot\frac{2}{k-1}\cdot\frac{\psi(k)}{\psi(k)+k-2}\cdot\ldots\cdot\frac{\psi(n)}{\psi(n)+n-2}\cdot \frac{p_n}{2(p_n+n)}\\
		&>\frac{2^{k-2}}{(k-1)!}\cdot\frac{\psi(k)}{\psi(k)+k}\cdot\ldots\cdot\frac{\psi(n)}{\psi(n)+n}\cdot\frac{\psi(n)}{2(\psi(n)+n)}\\
		&>\frac{2^{k-3}}{(k-1)!}\cdot\left(\prod_{j=k}^{\infty}\left(1+\frac{j}{\psi(j)}\right)\right)^{-1}\cdot\min_{n\in\mathbb{N}}\left\{\frac{\psi(n)}{\psi(n)+n}\right\}.
	\end{align*}
	The condition $\sum_{n=1}^{\infty}\frac{n}{\psi(n)}<\infty$ implies 
	$$0<\prod_{j=k}^{\infty}\left(1+\frac{j}{\psi(j)}\right)<\infty.$$
	Since $0<\frac{\psi(n)}{\psi(n)+n}<1$ and $\frac{\psi(n)}{\psi(n)+n}\to 1$ as $n\to\infty$, it follows that the minimum $\min_{n\in\mathbb{N}}\left\{\frac{\psi(n)}{\psi(n)+n}\right\}$ exists and is strictly positive. Hence, in both cases, ratios $$\frac{|\mathcal{T}(\Delta^{E}_{c_1 \ldots c_n})|}{|\Delta^{E}_{c_1 \ldots c_n}|}\qquad\text{and}\qquad\frac{|\mathcal{T}(U_i)|}{|U_i|}$$
	are bounded from below by a positive constant $m_k$ that does not depend on $n$. 
	
	Therefore,
	$$m_k^\alpha\cdot H^\alpha\left(E^k,\mathfrak{P}_{E}\right)<H^\alpha\left(\mathcal{T}\left(E^k\right),\mathfrak{P}_{E_\mathrm{mod}}\right)<2^\alpha H^\alpha\left(E^k,\mathfrak{P}_{E}\right).$$
	From these bounds, we deduce
	$$\dim_H\left(\mathcal{T}(E^k),\mathfrak{P}_{E_\mathrm{mod}}\right)=\dim_H\left(E^k,\mathfrak{P}_{E}\right)$$
	and
	$$\dim_H \mathcal{T}(E^k)=\dim_H E^k.$$
	Since $E=\bigcup_{k=1}^\infty E^k$, it follows that
	\begin{equation*}
		\dim_H \mathcal{T}(E) =\sup\left\{\dim_H \mathcal{T}(E^k)\right\}=\sup\left\{\dim_H E^k\right\}=\dim_H E.\qedhere
	\end{equation*}
\end{proof}

Thus, we obtain the following principle.

\begin{principle}[Fractal quasi-equivalence principle for Engel expansions.]
	Let $\psi\colon\mathbb{N}\to\mathbb{R}^+$ be a positive function satisfying $\sum_{n=1}^{\infty}\frac{n}{\psi(n)}<\infty$, and let $\mathfrak{M}$ be a subset of $\mathbb{N}^\mathbb{N}$ such that every sequence $(a_n)_{n=1}^\infty$ in $\mathfrak{M}$ satisfies
	$a_n\geq \psi(n)$ for all sufficiently large $n$. Then
	\begin{gather*}
		\dim_H \left\{x\in(0,1]\colon (p_n(x))_{n=1}^\infty\in\mathfrak{M}\right\}=\dim_H \left\{x\in(0,1]\colon (p'_n(x)-n+1)_{n=1}^\infty\in\mathfrak{M}\right\},
	\end{gather*}
	where $p_n(x)$ and $p'_n(x)$ denote the $n$th digits of the classical and modified Engel expansions of $x$, respectively.
\end{principle}

\subsection{Fractal quasi-equivalence principle for the Pierce expansion in Perron and traditional notations}

The modified Engel and Pierce expansions are particular cases of the positive and alternating Perron expansions, both defined by the sequence $P=\left(\varphi_n\right)_{n=0}^\infty$ with $\varphi_0=1$ and $\varphi_n(x_1,\ldots,x_n)=x_n$. 
As previously shown, the~transformation~$\mathcal{F}_P$ preserves the Hausdorff dimension. Note that series~\eqref{AltPerronSer} defines the Perron notation of the~Pierce expansion, which slightly differs from the traditional notation. Namely, the digits of the~Pierce expansion in the Perron notation exceed those in the traditional notation by one:
$$q_n(x)=\widetilde{q}_n(x)+1,$$
where $q_n(x)$ and $\widetilde{q}_n(x)$ denote the $n$th digits in the Perron and traditional notations, respectively. Consequently, a condition that holds for the sequence $(q_n(x))_{n=1}^\infty$ may fail to hold for the sequence $(\widetilde{q}_n(x))_{n=1}^\infty$, and vice versa. In the Perron notation for the Pierce expansion, we use the following conventions: the Pierce expansion of $x$ is denoted by $\Delta^\mathrm{Pierce}_{c_1 c_2\ldots}$; the Pierce cylinder of rank $k$ with base $c_1\ldots c_k$ is denoted by $\Delta^\mathrm{Pierce}_{c_1 \ldots c_k}$; and the faithful family $\mathfrak{P}^-$ is denoted by $\mathfrak{P}^-_\mathrm{Pierce}$.

Let $\mathfrak{M}\subset\mathbb{N}^\mathbb{N}$. In general, $$\left\{x\in(0,1)\setminus\mathbb{Q}\colon (q_n(x))_{n=1}^\infty\in\mathfrak{M}\right\}\not=\left\{x\in(0,1)\setminus\mathbb{Q}\colon (\widetilde{q}_n(x))_{n=1}^\infty\in\mathfrak{M}\right\}.$$
Therefore, Theorem~\ref{preserveTh2} alone does not suffice to establish analogies between the modified Engel expansion and the Pierce expansion in the traditional notation. To partially bridge this gap, we introduce the function $\mathcal{G}\colon(0,1)\setminus\mathbb{Q}\to(0,1)\setminus\mathbb{Q}$ defined by
$$\mathcal{G}(\Delta^\text{Pierce}_{c_1 c_2\ldots})=\Delta^\text{Pierce}_{(c_1+1) (c_2+1)\ldots},$$
that is, if $\mathcal{G}(x)=x'$, then $q_n(x')=q_n(x)+1$ for all $n\in\mathbb{N}$. From the definition of $\mathcal{G}$, it follows that:
\begin{itemize}
	\item $\widetilde{q}_n(x')=q_n(x)$ for all $n\in\mathbb{N}$.
	\item $\mathcal{G}\left(\Delta^\text{Pierce}_{c_1\ldots c_k}\right)=\Delta^\text{Pierce}_{(c_1+1)\ldots (c_k+1)}$;
	\item if $U\in\mathfrak{P}^-_\text{Pierce}$, then $\mathcal{G}(U)\in\mathfrak{P}^-_\text{Pierce}$.
\end{itemize}

For a positive function $\psi\colon \mathbb{N}\to\mathbb{R}$, we define the set $$\mathfrak{B}_\psi=\left\{x\in(0,1)\setminus\mathbb{Q}\colon q_n(x)\geq\psi(n)\text{ ~for all sufficiently large } n\right\}.$$

\begin{theorem}\label{preserveTh3}
	If $\sum_{n=1}^\infty\frac{1}{\psi(n)}<\infty$, then for every set $E\subset\mathfrak{B}_\psi$, we have $$\dim_H\mathcal{G}(E)=\dim_H E.$$
\end{theorem}

The proof follows the same scheme as in Theorem~\ref{preserveTh2}. The weaker condition on $\psi$ here arises from the fact that $\mathcal{G}$ increases each digit of the Pierce expansion by a constant independent of $n$.

Define the sets $D$ and $\widetilde{D}$ by
\begin{gather*}
	D=\left\{x\in(0,1)\setminus\mathbb{Q}\colon (q_n(x))_{n=1}^\infty\in\mathfrak{M}\right\},\\
	\widetilde{D}=\left\{x\in(0,1)\setminus\mathbb{Q}\colon(\widetilde{q}_n(x))_{n=1}^\infty\in\mathfrak{M}\right\}.
\end{gather*}
In general, $\mathcal{G}(D)\subseteq\widetilde{D}$. Indeed, if $x\in D$, then $(q_n(x))_{n=1}^\infty\in\mathfrak{M}$, so $(\widetilde{q}_n(\mathcal{G}(x)))_{n=1}^\infty\in\mathfrak{M}$, and hence $\mathcal{G}(x)\in\widetilde{D}$. However, if there exists $x'$ with $(\widetilde{q}_n(x'))_{n=1}^\infty\in\mathfrak{M}$ and $\widetilde{q}_1(x')=1$, then $x'$ cannot be obtained as $\mathcal{G}(x)$ for any $x$.

\begin{corollary}
	If $\sum_{n=1}^\infty\frac{1}{\psi(n)}<\infty$, then for every set $D\subseteq\mathfrak{B}_\psi$, we have
	$$\dim_H \widetilde{D}\geq \dim_H D.$$
\end{corollary}

\begin{lemma}\label{preserveLemma}
	If $\widetilde{q}_1(x)\geq 2$ for all $x\in\widetilde{D}$, then $\mathcal{G}(D)=\widetilde{D}$.
\end{lemma}

This lemma follows from the fact that for every strictly increasing sequence $(c_n)_{n=1}^\infty$ of natural numbers with $c_1\geq 2$, there exists a unique number $x\in(0,1)\setminus\mathbb{Q}$ such that $q_n(x)=c_n$ for all $n\in\mathbb{N}$.

\begin{corollary}\label{preserveCor3}
	If $\sum_{n=1}^\infty\frac{1}{\psi(n)}<\infty$, $D\subseteq\mathfrak{B}_\psi$, and  $\widetilde{q}_1(x)\geq 2$ for all $x\in\widetilde{D}$, then $$\dim_H\widetilde{D}=\dim_H D.$$
\end{corollary}

Thus, we obtain the following principle.

\begin{principle}[Fractal quasi-equivalence principle for the Pierce expansion in the Perron and traditional notations.]
	Let $\psi\colon\mathbb{N}\to\mathbb{R}^+$ be a positive function satisfying $\sum_{n=1}^{\infty}\frac{1}{\psi(n)}<\infty$, and let $\mathfrak{M}$ be a subset of $\mathbb{N}^\mathbb{N}$ such that every sequence $(a_n)_{n=1}^\infty$ in $\mathfrak{M}$ satisfies
	$a_n\geq \psi(n)$ for all sufficiently large $n$, and $a_1\geq 2$. Then
	\begin{gather*}
		\dim_H \left\{x\in(0,1)\setminus\mathbb{Q}\colon (q_n(x))_{n=1}^\infty\in\mathfrak{M}\right\}=\dim_H \left\{x\in(0,1)\setminus\mathbb{Q}\colon (\widetilde{q}_n(x))_{n=1}^\infty\in\mathfrak{M}\right\},
	\end{gather*}
	where $q_n(x)$ and $\widetilde{q}_n(x)$ denote the $n$th digits of the Pierce expansion of $x$ in the Perron and traditional notations, respectively.
\end{principle}

\section{Explanation of known analogies via fractal principles}

In this section, we show how new fractal principles explain known analogies between the modified Engel and Pierce expansions and between the classical and modified Engel expansions. These principles not only explain why such analogies arise, but also demonstrate that some properties need not be proved independently: they follow directly from their analogues once combined with our results.

Throughout this section, $p_n(x)$ and $p'_n(x)$ denote the $n$th digits of the classical and modified Engel expansions of $x$, respectively. Similarly, $q_n(x)$ and $\widetilde{q}_n(x)$ denote the $n$th digits of the Pierce expansion of $x$ in the Perron and traditional notations, respectively.

\subsection{Explanation of known analogies between the modified Engel and Pierce expansions via the~fractal equivalence principle for the Perron expansions.}

\begin{analogy}
	In \cite{WW2006}, B. W. Wang and J. Wu investigated Oppenheim expansions and determined the Hausdorff dimension of certain sets defined by conditions on the digits of these expansions. For the modified Engel expansion (see \cite[Corollary 2.7]{WW2006}), they proved that
	$$\dim_H\left\{x\in(0,1]\colon\lim_{n\to\infty}\frac{\log p'_{n+1}(x)}{\log p'_n(x)}\right\}=\frac{1}{\alpha}$$
	for all $\alpha\in[1,\infty)$. In \cite{A2024}, M.W.~Ahn calculated the Hausdorff dimension of the analogous set for the Pierce expansion in the traditional notation:
	$$F(\alpha)=\left\{x\in(0,1]\colon\lim_{n\to\infty}\frac{\log \widetilde{q}_{n+1}(x)}{\log \widetilde{q}_n(x)}=\alpha\right\}.$$
	In particular (see \cite[Theorem 1.12]{A2024}), $\dim_H F(\alpha)=1/\alpha$ for $\alpha\in[1,\infty]$ with the convention $1/\infty=0$. For $\alpha\in[1,\infty)$, the theorem of Ahn follows directly from result of Wang and Wu in combination with Theorem~\ref{preserveTh1} (the fractal equivalence principle for the Perron expansions), since 
	$$\lim_{n\to\infty}\frac{\log \widetilde{q}_{n+1}(x)}{\log \widetilde{q}_n(x)}=\alpha\iff\lim_{n\to\infty}\frac{\log q_{n+1}(x)}{\log q_n(x)}=\alpha.$$
	
	The case $\dim_H F(\infty)=0$ follows from a known result on the Pierce expansion, and we include a short proof for completeness.
	
	In \cite{FengTan2020}, Y. Feng and B. Tan investigated the set
	$$A(\varphi)=\left\{x\in[0,1)\colon \widetilde{q}_n(x)\geq\varphi(n) \text{ for infinitely many }n\in\mathbb{N}\right\},$$
	and proved that if $$\liminf_{n\to\infty}\frac{\log\log\varphi(n)}{n}=\log d\in[0,\infty],$$ 
	then $\dim_H A(\varphi)=1/d$ with the convention $1/\infty=0$ (see \cite[Theorem 1.1]{FengTan2020}).
	
	Let $x\in F(\infty)$. For any $M\in\mathbb{N}$, there exists $k=k(x)\geq 2$ such that 
	$$\frac{\log\widetilde{q}_{n+1}(x)}{\log\widetilde{q}_n(x)}>M+1 \text{ ~for all } n\geq k.$$
	Hence, 
	$$\widetilde{q}_{n}(x)>(\widetilde{q}_{n-1}(x))^{M+1}>\cdots>(\widetilde{q}_k(x))^{(M+1)^{n-k}}\geq 2^{(M+1)^{n-k}}.$$
	For sufficiently large $n$, the inequality $2^{(M+1)^{n-k}}>2^{M^n}$ holds, implying $\widetilde{q}_{n}(x)>2^{M^n}$ for infinitely many $n$. Thus, $x\in A(\varphi_M)$ and $F(\infty)\subseteq A(\varphi_M)$ with $\varphi_M(n)=2^{M^n}$. Moreover,  
	$$\liminf_{n\to\infty}\frac{\log\log\varphi(n)}{n}=\log M,$$
	so $\dim_H A(\varphi_M)=1/M$, and hence $\dim_H F(\infty)\leq 1/M$. Since $M$ is arbitrary, $\dim_H F(\infty)=0$.
	
	We note that in \cite{A2024} the dimension of $F(\infty)$ is established in Lemma 4.11 via a substantially more intricate argument, involving the construction of specific covers and estimates of the $\alpha$-Hausdorff measure. Our approach is shorter and, we believe, clearer.
\end{analogy}

\begin{analogy}
	In \cite{Wu2003}, J.~Wu calculated the Hausdorff dimension of certain sets defined by conditions on the~digit sequences of Oppenheim expansions. For the modified Engel expansion (see \cite[Corollary 3]{Wu2003}), Wu proved that the set
	$$\left\{x\in(0,1]\colon \lim_{n\to\infty}\frac{p'_{n+1}(x)}{p'_n(x)}=\alpha\right\}$$
	has Hausdorff dimension $1$ for all $\alpha\in[1,\infty)$. In \cite{A2024}, M.W.~Ahn determined the Hausdorff dimension of an~analogous set for the Pierce expansion in the traditional notation:
	$$B(\alpha)=\left\{x\in(0,1]\colon\lim_{n\to\infty}\frac{\widetilde{q}_{n+1}(x)}{\widetilde{q}_n(x)}=\alpha\right\}.$$
	In particular (see \cite[Theorem 1.8]{A2024}), $\dim_H B(\alpha)=1$ for all $\alpha\in[1,\infty]$.
	
	As in Analogy~1, we obtain that for $\alpha\in[1,\infty)$, the theorem of Ahn follows directly from the result of Wu in combination with the fractal equivalence principle for the Perron expansions. The inclusion $F(\alpha)\subset B(\infty)$ for all $\alpha>1$ implies that $\dim_H B(\infty)=1$, where $F(\alpha)$ is defined as in Analogy~1.
\end{analogy}

\begin{analogy}
	In \cite{ShangWu2021_1}, L. Shang and M. Wu investigated the exponent of convergence $\lambda(x)$ of $E$-digit sequence $(p_n(x))_{n=1}^\infty$, defined by
	$$\lambda(x)=\inf\left\{s\geq0\colon \sum_{n=1}^{\infty}\frac{1}{(p_n(x))^s}<\infty\right\}.$$
	In particular \cite[Theorem 4.1]{ShangWu2021_1}, they proved that
	\begin{equation*}
		\dim_H \left\{x\in(0,1]\colon\lambda(x)=\alpha\right\}
		=\dim_H \left\{x\in(0,1]\colon\lambda(x)\geq\alpha\right\}=
		\begin{cases}
			\begin{aligned}
				&1-\alpha, && 0\leq\alpha\leq 1; \\
				&0, &&  1<\alpha\leq \infty.
			\end{aligned}
		\end{cases}
	\end{equation*}
	We remark that for rational numbers we employ their infinite Engel expansions, whereas Shang and Wu consider only the finite analogue. Since rational numbers do not affect the Hausdorff dimension, this distinction is immaterial.
	
	Define the sets $S_\text{div}^E$ and $S_\text{div}^{E_\text{mod}}$ by
	\begin{gather*}
		S_\text{div}^E=\left\{x\in(0,1]\colon \sum_{n=1}^{\infty}\frac{1}{p_n(x)}=\infty\right\},\\
		S_\text{div}^{E_\text{mod}}=\left\{x\in(0,1]\colon \sum_{n=1}^{\infty}\frac{1}{p'_n(x)}=\infty\right\}.
	\end{gather*}
	Observe that $S_\text{div}^E\subseteq \left\{x\in(0,1]\colon\lambda(x)\geq 1\right\}$. Hence $\dim_H S_\text{div}^E=0$. Consider also 
	$$\mathcal{T}\left(S_\text{div}^E\right)=\left\{x\in(0,1]\colon \sum_{n=1}^{\infty}\frac{1}{p'_n(x)-n+1}=\infty \right\}.$$
	Note that $S_\text{div}^{E_\text{mod}}\subseteq \mathcal{T}\left(S_\text{div}^E\right)$. Corollary~\ref{ineqdim} implies that $\dim_H S_\text{div}^{E_\text{mod}}=0$. By Theorem~\ref{preserveTh1} (the fractal equivalence principle for the Perron expansions), we conclude that
	\begin{align*}
		\dim_H \left\{x\in(0,1)\setminus\mathbb{Q}\colon \sum_{n=1}^{\infty}\frac{1}{q_n(x)}=\infty\right\}
		=\dim_H \left\{x\in(0,1)\setminus\mathbb{Q}\colon \sum_{n=1}^{\infty}\frac{1}{\widetilde{q}_n(x)}=\infty\right\}=\dim_H S_\text{div}^{E_\text{mod}}=0.
	\end{align*}
	
	This result was previously established by Ahn (see \cite[Corollary~1.15]{A2025}) while studying the convergence exponent of Pierce expansion digit sequences. We also note that in the first arXiv version of \cite{A2025}, Ahn proved this result without using the convergence exponent.
\end{analogy}

\subsection{Explanation of known analogies between the classical Engel and Pierce expansions via the fractal quasi-equivalence principles for the Engel and Pierce expansions.}

\begin{analogy}
	In \cite{ShangWu2021}, L. Shang and M. Wu considered the set
	$$F_\psi=\left\{x\in(0,1]\colon \lim_{n\to\infty}\frac{\log\Delta_n(x)}{\psi(n)}=1\right\},$$
	where $\Delta_n:=p_{n}(x)-p_{n-1}(x)$ with $\Delta_1(x)=p_1(x)$ and $\psi(n)\colon\mathbb{N}\to\mathbb{R}^+$ is a non-decreasing function such that $\lim\limits_{n\to\infty}\frac{\psi(n)}{\log n}=\infty$. In particular \cite[Theorem~4.1]{ShangWu2021}, the authors proved that 
	$$\dim_H F_\psi=\frac{1}{1+\zeta},\qquad\text{where}\quad \zeta=\limsup_{n\to\infty}\frac{\psi(n+1)}{\psi(1)+\cdots+\psi(n)}.$$

	Consider the analogous set for the modified Engel expansion:
	$$F'_\psi=\left\{x\in(0,1]\colon \lim_{n\to\infty}\frac{\log\Delta'_n(x)}{\psi(n)}=1\right\},$$
	where $\Delta'_n:=p'_{n}(x)-p'_{n-1}(x)$ with $\Delta'_1(x)=p'_1(x)$, and $\psi$ as above.
	
	Let $x\in F_\psi$ and $x'=\mathcal{T}(x)$. Since $\Delta'_n(x')=\Delta_n(x)+1$ for all $n\geq 2$ and $\Delta'_1(x')=\Delta_1(x)$, it follows that $F'_\psi=\mathcal{T}(F_\psi)$. By assumption, $\lim\limits_{n\to\infty}\frac{\psi(n)}{\log n}=\infty$. So $\psi(n)>4\log n$ and $$p_n(x)>p_{n-1}(x)+n^{4(1+\varepsilon_n(x))}>n^{4(1+\varepsilon_n(x))}$$ for all sufficiently large $n$, where $\varepsilon_n(x)\to 0$ as $n\to\infty$. Hence $p_n(x)>n^3$ for all sufficiently large $n$. It is readily verified that $F_\psi\subset\mathfrak{A}_{n^3}$. Therefore, using Theorem~\ref{preserveTh2}, it follows that $$\dim_H F'_\psi=\dim_H F_\psi=\frac{1}{1+\zeta}.$$
	
	Applying the fractal equivalence principle for Perron expansions, we obtain that the corresponding sets defined in terms of $q_n(x)$ and $\widetilde{q}_n(x)$ also have Hausdorff dimension $\frac{1}{1+\zeta}$.
	Therefore, Theorem~1.4 from \cite{LLSh2026} follows directly from the result of Shang and Wu in combination with the fractal (quasi-)equivalence principles.
\end{analogy}

\begin{analogy}\label{5}
	In \cite[Theorem 4.4]{FangWu2018}, L.~Fang and M.~Wu considered the set
	$$\widetilde{F}(\varphi)=\left\{x\in(0,1]\colon p_n(x)\geq\varphi(n) \text{ for all } n\in\mathbb{N}\right\},$$
	where $\varphi\colon\mathbb{N}\to\mathbb{R}^+$, and proved that $\dim_H \widetilde{F}(\varphi)=1/\gamma$ with the convention $1/\infty=0$, where $\gamma$ is given by $\log\gamma=\limsup\limits_{n\to\infty}\frac{\log\log \varphi(n)}{n}$.
	
	Consider the set $$\widetilde{F}_\text{mod}(\varphi)=\left\{x\in(0,1]\colon p'_n(x)\geq\varphi(n) \text{ for all } n\in\mathbb{N}\right\}.$$
	Assume that $\varphi(n)\geq n+1$, since $p'_n(x)\geq n+1$. Let $\varphi'\colon\mathbb{N}\to\mathbb{R}^+$ be given by $\varphi'(n)=\varphi(n)-n+1$ for all $n\in\mathbb{N}$. Then
	$$\limsup\limits_{n\to\infty}\frac{\log\log \varphi'(n)}{n}=\limsup\limits_{n\to\infty}\frac{\log\log \varphi(n)}{n},$$
	$\dim_H\widetilde{F}(\varphi')=\dim_H\widetilde{F}(\varphi)$, and $\widetilde{F}_\text{mod}(\varphi)=\mathcal{T}\left(\widetilde{F}(\varphi')\right)$. 
	Corollary \ref{ineqdim} implies that
	$$\dim_H \widetilde{F}_\text{mod}(\varphi)\leq\dim_H \widetilde{F}(\varphi')=\dim_H \widetilde{F}(\varphi).$$
	If $\limsup\limits_{n\to\infty}\frac{\log\log \varphi(n)}{n}=\infty$, then $\dim_H \widetilde{F}_\text{mod}(\varphi)= \dim_H \widetilde{F}(\varphi)=0.$
	
	Assume that 
	$\lambda=\limsup\limits_{n\to\infty}\frac{\log\log \varphi(n)}{n}<\infty$ and define a function $\omega\colon\mathbb{N}\to\mathbb{R}^+$  by
	$$\omega(n)=\max\left\{\varphi'(n),e^{ne^{\lambda n}}\right\}.$$
	Then $\widetilde{F}(\varphi')\supset \widetilde{F}(\omega)$ and $\widetilde{F}_\text{mod}(\varphi)\supset \mathcal{T}\left(\widetilde{F}(\omega)\right)$, and hence
	$$\dim_H  \widetilde{F}_\text{mod}(\varphi)\geq \dim_H \mathcal{T}\left(\widetilde{F}(\omega)\right).$$
	Moreover, $\sum_{n=1}^{\infty}\frac{n}{\omega(n)}<\infty$ and $\limsup\limits_{n\to\infty}\frac{\log\log \omega(n)}{n}=\lambda$. Using the fractal quasi-equivalence principle for Engel expansions, we obtain 
	\begin{equation*}
		\dim_H \mathcal{T}\left(\widetilde{F}(\omega)\right)= \dim_H \widetilde{F}(\omega)=\dim_H \widetilde{F}(\varphi).
	\end{equation*}
	Hence, $\dim_H  \widetilde{F}_\text{mod}(\varphi)=\dim_H \widetilde{F}(\varphi)$.
	
	Consider the set 
	$$\widetilde{F}^*_\text{mod}(\varphi)=\left\{x\in(0,1]\colon p'_n(x)\geq\varphi(n) \text{ for sufficiently large } n\in\mathbb{N}\right\}.$$
	Since $\widetilde{F}^*_\text{mod}(\varphi)=\bigcup_{k=1}^\infty\widetilde{F}_\text{mod}(\varphi_k)$, where $\varphi_k$ given by
	$$\varphi_k(n)=\begin{cases}
		n+1,\quad&\text{if }n<k,\\
		\varphi(n),\quad&\text{if }n\geq k,
	\end{cases}$$
	then $\dim_H \widetilde{F}^*_\text{mod}(\varphi)=\sup\left\{\dim_H \widetilde{F}_\text{mod}(\varphi_k)\right\}=\dim_H \widetilde{F}_\text{mod}(\varphi)=\dim_H\widetilde{F}(\varphi)$.
	
	Applying the fractal equivalence principle for Perron expansions, we obtain that the corresponding sets defined in terms of $q_n(x)$ and $\widetilde{q}_n(x)$ with the conditions ''for all $n$'' and ``for all sufficiently large $n$'' also have Hausdorff dimension $1/\gamma$. Thus, Theorem~1.2 in \cite{FengTan2020} and Theorem~1.1 in \cite{LLL2025_1} follow directly from the results of Fang and Wu in combination with the fractal (quasi-)equivalence principles. In particular, this implies the result of Ahn from \cite[Theorem 1.6]{A2024}.
\end{analogy}

\begin{analogy}\label{6}
	In \cite[Theorem 4.2]{FangWu2018}, L.~Fang and M.~Wu considered the set
	$$F(\varphi)=\left\{x\in(0,1]\colon p_n(x)\geq\varphi(n) \text{ for infinitely many } n\in\mathbb{N}\right\},$$
	where $\varphi\colon\mathbb{N}\to\mathbb{R}^+$, and proved that $\dim_H F(\varphi)=1/\beta$, where $\beta$ is given by $\log\beta=\liminf\limits_{n\to\infty}\frac{\log\log \varphi(n)}{n}$.
	
	Consider the set $$F_\text{mod}(\varphi)=\left\{x\in(0,1]\colon p'_n(x)\geq\varphi(n) \text{ for infinitely many } n\in\mathbb{N}\right\}.$$
	As in analogy 5, we have
	$$\dim_H F_\text{mod}(\varphi)\leq \dim_H F(\varphi)=1/\beta.$$
	Moreover, if $\liminf\limits_{n\to\infty}\frac{\log\log \varphi(n)}{n}=\infty$, then $\dim_H F_\text{mod}(\varphi)= \dim_H F(\varphi)=0.$
	
	Let
	$\lambda=\liminf\limits_{n\to\infty}\frac{\log\log \varphi(n)}{n}<\infty$. For all $\varepsilon>0$, define a function $\omega_\varepsilon\colon\mathbb{N}\to\mathbb{R}^+$ given by
	$$\omega_\varepsilon(n)=e^{e^{(\lambda+\varepsilon) n}}.$$
	Since $\frac{\log\log \omega_\varepsilon(n)}{n}=\lambda+\varepsilon$, we have $\omega_\varepsilon(n)\geq \varphi(n)$ for infinitely many $n$ and $F_\text{mod}(\varphi)\supset \widetilde{F}(\omega_\varepsilon)$, where the set $\widetilde{F}(\cdot)$ is defined as in Analogy~5. Therefore, 
	$$\dim_H F_\text{mod}(\varphi)\geq \dim_H \widetilde{F}(\omega_\varepsilon)=\frac{1}{e^{\lambda+\varepsilon}}=\frac{1}{\beta\cdot e^\varepsilon}$$
	for all $\varepsilon>0$, and hence $\dim_H F_\text{mod}(\varphi)\geq 1/\beta$. Thus,
	$$\dim_H F_\text{mod}(\varphi)= \dim_H F(\varphi)=1/\beta.$$
	
	Applying the fractal equivalence principle for Perron expansions, we obtain that the corresponding sets defined in terms of $q_n(x)$ and $\widetilde{q}_n(x)$ with the condition ''for infinitely many $n$'' also have Hausdorff dimension $1/\beta$.	Thus, Theorem 1.1 in \cite{FengTan2020} and Theorem~1.2 in \cite{LLL2025_1} follow directly from the results of Fang and Wu in combination with the fractal (quasi-)equivalence principles.
\end{analogy}

\begin{analogy}
	In \cite[Theorems~4.9 and 4.10]{FangWu2018}, L.~Fang and M.~Wu proved that
	$$\dim_H\left\{x\in(0,1]\colon p_{n+1}(x)-p_n(x)\geq \varphi(n)\text{ for all }n\in\mathbb{N}\right\}=1/\gamma$$
	and
	$$\dim_H\left\{x\in(0,1]\colon p_{n+1}(x)-p_n(x)\geq \varphi(n)\text{ for infinitely many }n\in\mathbb{N}\right\}=1/\beta$$
	where $\gamma$ and $\beta$ are defined as above in Analogies~5 and~6.
	
	Arguing as in the two previous analogies, we obtain that the corresponding sets defined in terms of $p'_n(x)$, $q_n(x)$, and $\widetilde{q}_n(x)$ with the conditions ``for all $n$'' and ``for all sufficiently large $n$'' have Hausdorff dimension $1/\gamma$, whereas those defined with the condition ``for infinitely many $n$'' have Hausdorff dimension $1/\beta$. Thus, Theorem~1.3 and Theorem~1.4 in \cite{LLL2025_1} follow directly from the result Fang and Wu in combination with the fractal (quasi-)equivalence principles.
\end{analogy}

\begin{analogy}
	In \cite[Theorems~4.13 and~4.14]{FangWu2018}, L.~Fang and M.~Wu considered the sets
	$$\widetilde{R}(\varphi)=\left\{x\in(0,1]\colon \frac{p_{n+1}(x)}{p_n(x)}\geq \varphi(n)\text{ for all }n\in\mathbb{N}\right\},$$
	and
	$$R(\varphi)=\left\{x\in(0,1]\colon \frac{p_{n+1}(x)}{p_n(x)}\geq \varphi(n)\text{ for infinitely many }n\in\mathbb{N}\right\}.$$
	In particular, they proved that $\dim_H \widetilde{R}(\varphi)=1/\gamma$ and $\dim_H R(\varphi)=1/\beta$, where $\gamma$ and $\beta$ are defined as above in Analogies~5 and~6.
	
	Consider the sets
	\begin{gather*}
		\widetilde{R}_\text{mod}(\varphi)=\left\{x\in(0,1]\colon \frac{p'_{n+1}(x)}{p'_n(x)} \geq \varphi(n)\text{ for all }n\in\mathbb{N}\right\},\\
		\mathcal{T}\left(\widetilde{R}(\varphi)\right)=\left\{x\in(0,1]\colon \frac{p'_{n+1}(x)-n}{p'_n(x)-n+1}\geq \varphi(n)\text{ for all }n\in\mathbb{N}\right\}.
	\end{gather*}
	Denote $\varphi_1(n)=(n+1)\varphi(n)$ and $\varphi_2(n)=\max\left\{1,\varphi(n)-1\right\}$.
	Since
	\begin{gather*}
		\frac{p'_{n+1}(x)-n}{p'_n(x)-n+1}\geq \varphi_1(n) \Longrightarrow \frac{p'_{n+1}(x)}{p'_{n}(x)}\geq \varphi(n) \Longrightarrow \frac{p'_{n+1}(x)-n}{p'_n(x)-n+1}\geq \varphi_2(n),
	\end{gather*}
	then
	$$\mathcal{T}\left(\widetilde{R}(\varphi_1)\right) \subseteq\widetilde{R}_\text{mod}\left(\varphi\right)\subseteq \mathcal{T}\left(\widetilde{R}(\varphi_2)\right)$$
	and
	$$\dim_H\mathcal{T}\left(\widetilde{R}(\varphi_1)\right) \leq\dim_H \widetilde{R}_\text{mod}\left(\varphi\right)\leq \dim_H \mathcal{T}\left(\widetilde{R}(\varphi_2)\right).$$
	Using the same argument as for the set $\widetilde{F}_\text{mod}(\varphi)$ in Analogy~\ref{5}, we obtain 
	$$\dim_H \mathcal{T}\left(\widetilde{R}(\varphi)\right)=\dim_H \widetilde{R}(\varphi).$$
	Moreover,
	$$\limsup_{n\to\infty}\frac{\log\log \varphi_1}{n}=\limsup_{n\to\infty}\frac{\log\log \varphi_2}{n}=\limsup_{n\to\infty}\frac{\log\log \varphi}{n}.$$
	Thus, 
	$\dim_H \widetilde{R}_\text{mod}\left(\varphi\right)= \dim_H \mathcal{T}\left(\widetilde{R}(\varphi_1)\right) =\dim_H \mathcal{T}\left(\widetilde{R}(\varphi_2)\right) =1/\gamma.$
	
	Arguing as in Analogies 5 and 6, we obtain that the corresponding sets defined in terms of $p'_n(x)$, $q_n(x)$, and $\widetilde{q}_n(x)$ with the conditions ``for all $n$'' and ``for all sufficiently large $n$'' have Hausdorff dimension $1/\gamma$, whereas those defined with the condition ``for infinitely many $n$'' have Hausdorff dimension $1/\beta$. Thus, Theorem~1.5 and Theorem~1.6 in \cite{LLL2025_1} follow directly from the result Fang and Wu in combination with the fractal (quasi-)equivalence principles.
\end{analogy}

\begin{analogy}
	In \cite[Theorem 3.1]{ShangWu2021} and \cite[Theorem 3.2]{ShangWu2020}, L.~Shang and M.~Wu defined and investigated the set
	$$E_\varphi=\left\{x\in(0,1]\colon \lim_{n\to\infty}\frac{\log p_n(x)}{\varphi(n)}=1\right\},$$
	where $\varphi\colon\mathbb{N}\to\mathbb{R}^+$	is a non-decreasing function satisfying $\lim\limits_{n\to\infty}\varphi(n)=\infty$. Assume that $$\lim_{n\to\infty} \frac{\varphi(n)}{\log n}=\gamma\in[0,\infty]\qquad\text{and}\qquad\limsup_{n\to\infty}\frac{\varphi(n+1)}{\varphi(1)+\cdots+\varphi(n)}=\xi.$$
	Then they proved that
	\begin{equation*}
		\dim_H E_\varphi=
		\begin{cases}
			\begin{aligned}
				&0, && \text{if } \gamma\in [0,1), \\
				&1-\gamma^{-1}, && \text{if } \gamma\in [1,\infty), \\
				&(1+\xi)^{-1}, && \text{if } \gamma=\infty.
			\end{aligned}
		\end{cases}
	\end{equation*}
	An analogous set for the Pierce expansion in the traditional notation,
	$$\widetilde{E}_\varphi=\left\{x\in(0,1)\setminus\mathbb{Q}\colon \lim_{n\to\infty}\frac{\log \widetilde{q}_n(x)}{\varphi(n)}=1\right\},$$
	was investigated by M.W. Ahn in \cite[Theorem 1.1]{A2024}, where its Hausdorff dimension was also calculated. 
	
	Now consider the set
	$$\mathcal{F}\bigl(\mathcal{T}(E_\varphi)\bigr)=\left\{x\in(0,1)\setminus\mathbb{Q}\colon\lim_{n\to\infty}\frac{\log(q_n(x)-n+1)}{\varphi(n)}=1\right\}.$$
	If $\gamma>1$, then
	$$\lim_{n\to\infty}\frac{\log \widetilde{q}_n(x)}{\varphi(n)}=1\iff\lim_{n\to\infty}\frac{\log q_n(x)}{\varphi(n)}=1\iff\lim_{n\to\infty}\frac{\log(q_n(x)-n+1)}{\varphi(n)}=1.$$ 
	Thus, $\widetilde{E}_\varphi=\mathcal{F}\bigl(\mathcal{T}(E_\varphi)\bigr)$.
	
	Assume that $\gamma\in(2,\infty]$. Define the function $$\psi(n)= \exp\left(\frac{\gamma+6}{2(\gamma+2)}\cdot\varphi(n)\right);$$ 
	here $0<\frac{\gamma+6}{2(\gamma+2)}<1$. Then $$\varphi(n)>\frac{\gamma+2}{2}\log n, \qquad \psi(n)>n^\frac{\gamma+6}{4},\qquad\text{and}\qquad \frac{n}{\psi(n)}<\frac{1}{n^{\frac{\gamma+2}{4}}}$$ for all sufficiently large $n$. Therefore, $\sum_{n=1}^\infty \frac{n}{\psi(n)}<\infty$. Moreover, for any $x\in E_\varphi$, we have $$p_n(x)=e^{\varphi(n)(1+\varepsilon_n(x))}>\psi(n)$$
	for all sufficiently large $n$,
	where $\varepsilon_n(x)\to 0$ as $n\to\infty$. Consequently, $x\in\mathfrak{A}_\psi$, and thus $E_\varphi\subset\mathfrak{A}_\psi$. Then Theorems~\ref{preserveTh1} and~\ref{preserveTh2} (fractal (quasi-)equivalence principles for Perron and Engel expansions) imply that
	$$\dim_H \widetilde{E}_\varphi=\dim_H E_\varphi=1-\frac{1}{\gamma}.$$
	In the case $\gamma=\infty$, an analogous argument yields $$\dim_H \widetilde{E}_\varphi=\dim_H E_\varphi=\frac{1}{1+\xi}.$$
	Hence, for $\gamma\in(2,\infty]$, the result of Ahn follows from the result of Shang and Wu by Theorems~\ref{preserveTh1} and~\ref{preserveTh2}, whereas our method does not directly apply in the case $\gamma\in[1,2)$.
\end{analogy}

\begin{analogy}
	In \cite[Theorem~4.2]{FangShang2024}, L. Fang and L. Shang considered the set
	$$D_\varphi(\alpha,\beta)=\left\{x\in(0,1]\colon \liminf_{n\to\infty}\frac{\log p_n(x)-n}{\varphi(n)}=\alpha, \limsup_{n\to\infty}\frac{\log p_n(x)-n}{\varphi(n)}=\beta\right\},$$
	where $\varphi\colon \mathbb{N}\to\mathbb{R}^+$ is an increasing function satisfying
	\begin{equation}\label{condphi2}
		\lim_{n\to\infty}\varphi(n)=\infty\qquad\text{and}\qquad \lim_{n\to\infty} \bigl(\varphi(n+1)-\varphi(n)\bigr)=0.
	\end{equation}
	They proved that $D_\varphi(\alpha,\beta)$ has full Hausdorff dimension for all $\alpha,\beta\in[-\infty,+\infty]$ with $\alpha\leq \beta$. An analogous set for the Pierce expansion in the traditional notation was studied by C.~Long, L.~Lu, and Y.~Liang in \cite[Theorem~1.1]{LLL2025_2}, where it was shown that this set also has full Hausdorff dimension.
	
	Let $x\in D_\varphi(\alpha,\beta)$ with $\alpha\in \mathbb{R}$ and $\beta\geq\alpha$. Under the conditions \eqref{condphi2},
	\begin{equation}\label{1}
		p_n(x)\geq e^{n+(\alpha-1)\varphi(n)}\geq 2^n=\psi(n)
	\end{equation}
	for all sufficiently large $n$ and  $D_\varphi(\alpha,\beta)\subset \mathfrak{A}_{2^n}$. Therefore, using the fractal quasi-equivalence principle for Engel expansions, we obtain the set
	\begin{equation*}
		\mathcal{T}\left(D_\varphi(\alpha,\beta)\right)=\left\{x\in(0,1]\colon \liminf_{n\to\infty}\frac{\log \log\bigl(p'_n(x)-n+1\bigr)}{\varphi(n)}=\alpha,~  \limsup_{n\to\infty}\frac{\log \bigl(p'_n(x)-n+1\bigr)-n}{\varphi(n)}=\beta\right\}
	\end{equation*}
	has full Hausdorff dimension. If $\alpha=+\infty$, the arguments are similar and $\dim_H D'_\varphi(+\infty,+\infty)=1$.
	
	Define the set
	$$D'_\varphi(\alpha,\beta)=\left\{x\in(0,1]\colon \liminf_{n\to\infty}\frac{\log p'_n(x)-n}{\varphi(n)}=\alpha, \limsup_{n\to\infty}\frac{\log p'_n(x)-n}{\varphi(n)}=\beta\right\}.$$
	If $x\in D_\varphi(\alpha,\beta)$ or $x\in D'_\varphi(\alpha,\beta)$ with $\alpha>-\infty$, then $p'_n(x)\geq \psi(n)$ for all sufficiently large $n$ and
	$$\lim_{n\to\infty}\left(\log p'_n(x)-\log \bigl(p'_n(x)-n+1\bigr)\right)=0.$$
	Thus, assuming $\alpha>-\infty$, we conclude that $\mathcal{T}\bigl(D_\varphi(\alpha,\beta)\bigr)=D'_\varphi(\alpha,\beta)$ and $$\dim_H D'_\varphi(\alpha,\beta)=1.$$
	Moreover, using the fractal equivalence principles for Perron expansions, we obtain that the corresponding sets defined in terms of $q_n(x)$ and $\widetilde{q}_n(x)$ also have full Hausdorff dimension. Hence, for $\alpha>-\infty$, the result of Long, Lu and Liang follows from the result of Fang and Shang, whereas our method does not directly apply in the case $\alpha=-\infty$.
\end{analogy}

\begin{remark}
	To the best of our knowledge, all intermediate results obtained in Analogies 4--10 for the modified Engel expansions are presented here for the first time. In particular, they can be derived not only from the properties of the classical Engel expansions, but also from the corresponding properties of the Pierce expansions in combination with the fractal equivalence principle for the Perron expansions.
\end{remark}

\section{New analogies between the classical Engel, modified Engel, and Pierce expansions}

Throughout this section, $p_n(x)$ and $p'_n(x)$ denote the $n$th digits of the classical and modified Engel expansion of $x$, respectively. Similarly, $q_n(x)$ and $\widetilde{q}_n(x)$ denote the $n$th digits of the Pierce expansion of $x$ in the Perron and traditional notations, respectively.

\begin{theorem}
	Let $\varphi\colon\mathbb{N}\to\mathbb{R}^+$ be a function satisfying $\lim\limits_{n\to\infty}\varphi(n)=\infty$, and let $R'_n(x)=\frac{p'_n(x)}{p'_{n-1}(x)}$. Then 
	$$\dim_H \left\{x\in(0,1]\colon \lim_{n\to\infty}\frac{\log R'_n(x)}{\varphi(n)}=1\right\}=\frac{1}{1+\theta},$$
	where
	$$\theta=\limsup_{n\to\infty}\frac{\sum_{k=1}^{n+1}\varphi(k)}{\sum_{k=1}^{n}(n-k+1)\varphi(k)}.$$
\end{theorem}

\begin{proof}
	In \cite[Theorem 5.1]{ShangWu2021}, L. Shang and M. Wu considered the set
	$$R(\varphi)=\left\{x\in(0,1]\colon \lim_{n\to\infty}\frac{\log R_n(x)}{\varphi(n)}=1\right\},$$
	where $\varphi\colon\mathbb{N}\to\mathbb{R}^+$ is a function satisfying $\lim\limits_{n\to\infty}\varphi(n)=\infty$, $R_1(x)=p_1(x)$, and $R_n(x)=\frac{p_n(x)}{p_{n-1}(x)}$ for $n\geq2$. They proved that $\dim_H R(\varphi)=\frac{1}{1+\theta}$, where $\theta$ as above.
	If $x\in R(\varphi)$, then $p_n(x)\geq 3p_{n-1}(x)$ and $p_n(x)\geq 2^n$ for all sufficiently large $n$. Hence $R(\varphi)\subseteq\mathfrak{A}_{2^n}$. Since
	$$\lim_{n\to\infty}\frac{\log R_n(x)}{\varphi(n)}=1\iff\lim_{n\to\infty}\frac{\log R'_n\bigl(\mathcal{T}(x)\bigr)}{\varphi(n)}=1,$$
	we have
	$$\mathcal{T}\bigl(R(\varphi)\bigr)=\left\{x\in(0,1]\colon \lim_{n\to\infty}\frac{\log R'_n(x)}{\varphi(n)}=1\right\}$$
	and hence $\dim_H \mathcal{T}\bigl(R(\varphi)\bigr)=\dim_H R(\varphi)=\frac{1}{1+\theta}$. 
\end{proof}

\begin{corollary}
	Let $\varphi\colon\mathbb{N}\to\mathbb{R}^+$ be a function satisfying $\lim\limits_{n\to\infty}\varphi(n)=\infty$. Then the sets
	$$\left\{x\in(0,1)\setminus\mathbb{Q}\colon \lim_{n\to\infty}\frac{\log\frac{q_n(x)}{q_{n-1}(x)}}{\varphi(n)}=1\right\}$$
	and
	$$\left\{x\in(0,1)\setminus\mathbb{Q}\colon \lim_{n\to\infty}\frac{\log\frac{\widetilde{q}_n(x)}{\widetilde{q}_{n-1}(x)}}{\varphi(n)}=1\right\}$$
	coincide and have Hausdorff dimension $\frac{1}{1+\theta}$, where $\theta$ as above.
\end{corollary}

\begin{theorem}\label{6.9}
	For all $k\in(1,\infty)$,
	$$\dim_H \left\{x\in(0,1]\colon \frac{p'_{n+1}(x)}{p'_n(x)}\leq k\text{ ~for all }n\in\mathbb{N}\right\}=1.$$
\end{theorem}

\begin{proof}
	In \cite[Corollary~1.10]{A2024}, M.W. Ahn proved that
	$$\dim_H \left\{x\in(0,1)\setminus\mathbb{Q}\colon \frac{\widetilde{q}_{n+1}(x)}{\widetilde{q}_n(x)}\leq k\text{ ~for all }n\in\mathbb{N}\right\}=1$$ 
	for all $k\in(1,\infty)$. Since $\frac{q_{n+1}(x)}{q_{n}(x)}<\frac{\widetilde{q}_{n+1}(x)}{\widetilde{q}_n(x)}$, it follows that
	$$\dim_H \left\{x\in(0,1)\setminus\mathbb{Q}\colon \frac{q_{n+1}(x)}{q_n(x)}\leq k\text{ ~for all }n\in\mathbb{N}\right\}=1.$$ 
	Applying Theorem~\ref{preserveTh1} (the fractal equivalence principle for the Perron expansions) completes the proof.
\end{proof}

Theorem~\ref{6.9} extends that part of the result of Wang and Wu from \cite[Example~4]{WW2007} which concerns the modified Engel expansion.

\section*{Acknowledgements}
We thank the anonymous referees for their valuable comments and suggestions, which have significantly improved the clarity and structure of the paper and helped reveal additional applications of our results. This work was supported by a grant from the Simons Foundation (SFI-PD-Ukraine-00014586, M.M.).

\end{document}